\documentclass{amsart}
\usepackage{amsthm}
\usepackage[utf8]{inputenc}
\usepackage{appendix}
\usepackage{cite}
\usepackage{amssymb,amsmath,amsthm}

\newtheorem{theorem}{Theorem}[section]
\newtheorem{lemma}[theorem]{Lemma}

\theoremstyle{definition}

\theoremstyle{remark}

\numberwithin{equation}{section}

\DeclareTextCommandDefault{\textcopyright}{\textcircled{c}}
\DeclareTextCommandDefault{\textregistered}{\textcircled{%
      \check@mathfonts\fontsize\sf@size\z@\math@fontsfalse\selectfont R}}

\begin{document}

\title{Bloch's Theorem for Heat Maps}

\author{Jean C. Cortissoz}
\address{Department of Mathematics, Universidad de los Andes, Bogot\'a DC, COLOMBIA.}






\begin{abstract}
In this paper we give a proof via the contraction mapping principle
of a Bloch-type theorem for normalised Bochner-Takahashi $K$-mappings,
which are solutions to the homogenous equation $Lu=0$, where $L$ is the
heat operator. 
\end{abstract}

\maketitle
\section{Introduction}

\begin{theorem}
Let $f\left(z\right)$ be an analytic function on $B_1\left(0\right)$ (the unit disk centered at the origin)
satisfying $f\left(0\right)=0$ and $f'\left(0\right)=1$. Then there is a constant $B$ (called 
Bloch's constant) independent of $f$, such that there is a subdomain
$\Omega\subset B_1\left(0\right)$ where $f$ is one-to-one and whose
image contains a disk of radius $B$ (which we shall call, as is customary, a {\bf schlicht 
or univalent disk}, 
and we will say that $f$ covers a {\bf schlicht or univalent disk of radius} B).
\end{theorem}

The previous statement is known as Bloch's theorem and it was proved by Andr\'e Bloch in \cite{Bloch}
(Proposition G). Besides its beauty, 
Bloch's theorem is nothing short of surprising: who would have expected that a
bound from below for the radius of a univalent disk 
covered by a member of a family of holomorphic functions on the unit disk only depends on 
the normalisation at $z=0$, namely that $\left|f'\left(0\right)\right|=1$
(the fact that $f\left(0\right)=0$ is actually irrelevant)?
On the other hand, one of the most,
if not the most, celebrated problem in Geometric Function Theory is to find the exact value of
Bloch's constant $B$.

Starting with the work of Bochner \cite{Bochner}, Bloch's theorem has been generalised
to several real and complex variables. The work of Wu \cite{Wu} is 
of particular interest, as he proved a very general Bloch type theorem for
solutions to homogeneous hypoelliptic equations. Wu uses compactness arguments in his proofs,
and he does not give effective estimates on the radius of the schlicht (univalent) disks covered
by the different families of functions that he considers in his work
(we supply a new proof of Wu's result with estimates for elliptic operators of constant coefficients in \cite{Cortissoz}). Also, 
in Wu's work hypoelliptic
operators as the heat operator (our main concern in this paper) are not covered, as the 
derivatives in the operators considered by him need to be of the same order. To have a broader
overview on the subject, we invite the reader to consult the paper \cite{Chen2} and the references therein.

\subsection{Heat Bochner-Takahashi mappings}
A {\bf Heat Bochner-Takahashi $K$-mapping} on the unit ball $B_1\left(0\right)\subset \mathbb{R}^m$ is a map 
 \[
 F:B_1\left(0\right)\subset \mathbb{R}^{m+1}\longrightarrow \mathbb{R}^{m+1},
 \]
 \[
 F\left(x_1,\dots,x_m,t\right)=\left(F_1\left(x_1,\dots,x_m,t\right),\dots,F_m\left(x_1,\dots,x_m,t\right)\right)^T,
 \]
 such that
 each one of its components satisfies {\it the heat equation}, i.e., 
 \[
 \Delta F_j -\frac{\partial}{\partial t}F_j=0,
 \]
 and such that
 \begin{equation}
\label{Takahashicondition0}
\max_{\left\|z\right\|\leq r}\left\|F'\left(z\right)\right\|\leq K
\max_{\left\|z\right\|\leq r}\left|\mbox{det}\left(F'\left(z\right)\right)\right|^{\frac{1}{m+1}},
\quad \mbox{for all}\quad 0\leq r<1.
\end{equation}
Here, for an $m+1$-tuple $z=\left(x_1,\dots,x_m,t\right)$, 
$\left\|z\right\|$ denotes its euclidean norm, and for a matrix $A=\left(a_{ij}\right)$, $\left\|A\right\|$ denotes
the norm
\[
\left(\sum_{i,j} \left|a_{ij}\right|^2\right)^{\frac{1}{2}}.
\]

We have the following Bloch type theorem for Heat Bochner-Takahashi $K$-maps.
\begin{theorem}
\label{BlochTakahashi}
Let $F:B_1\left(0\right)\subset \mathbb{R}^{m+1}\longrightarrow\mathbb{R}^{m+1}$ be a 
Heat Bochner-Takahashi $K$-mapping, normalised so that
$\left|\mbox{det}\left(F'\left(0\right)\right)\right|=1$. Then 
$F$ covers a schlicht disk of radius at least $\dfrac{0.22^4}{2^{m+5}m}\dfrac{1}{a_m^4 K^{2m+3}}$,
where $a_m$ is a dimensional constant which depends only on the $m+1$-dimensional heat kernel.
\end{theorem}

The proof of this theorem follows the proof given in \cite{Cortissoz} with some modifications
-we use a Taylor expansion of order 2, instead of using the analiticity of the
functions in the family. We include
all the details below so that this paper can be read independently.

We must point out that Theorem \ref{BlochTakahashi} generalizes Bochner's theorem (which is for harmonic maps),
not only by considering a more general family of functions under a weaker condition
(instead of (\ref{Takahashicondition0}) Bochner considers a poinwise estimate), but
also by providing effective estimates from below for the radius of a univalent disk covered by
any member of the family.

\section{Preliminary notions and notation}

We will be studying functions $F:B_1\left(0\right)\longrightarrow \mathbb{R}^{m+1}$. We will
refer to $x_1,\dots,x_m$ as the spatial variables and to $x_{m+1}$ as the time variable,
which we will sometimes denote by a $t$. This means that our coordinates
are $(x,t)$ where $x=\left(x_1,\dots,x_m\right)$. We shall employ the letter $z$ to refer
to $\left(x,t\right)$.

The heat operator $L$ is defined as
\[
L= \sum_{j=1}^m\frac{\partial^2}{\partial x_m^2}-\frac{\partial}{\partial t}.
\]

Our notation for open balls has a little pecualiarity. If we write $B_r\left(x_0\right)$
we mean the ball of radius $r$ centered in $x_0$ in $\mathbb{R}^m$, whereas if we write
$B_r\left(x_0,t_0\right)$ we mean the ball of radius $r$ centered at $\left(x_0,t_0\right)$
in $\mathbb{R}^{m+1}$. Distances are measured in the euclidean metric, and we denote
the euclidean norm of $w\in \mathbb{R}^l$ by $\left\|w\right\|$. The closure
of a set $C$ will be denoted by $\overline{C}$.

Given a square $m+1$ by $m+1$ matrix $A=\left(a_{ij}\right)$, $\lambda\left(A\right)$ and $\Lambda\left(A\right)$
represent the square root of the minimum and the maximum of the eigenvalues
of $A^*A$. As we said before, $\left\|A\right\|$ represents the norm
\[
\left(\sum_{i,j} \left|a_{ij}\right|^2\right)^{\frac{1}{2}}
\]
of $A$. This norm satisfies the following well-known inequalities
\[
\left|A\right|\leq \left\|A\right\|\leq \sqrt{m+1}\left|A\right|,
\]
where $\left|A\right|$ is the operator norm of $A$. We also have the the following identities:
\[
\left|A^{-1}\right|=\Lambda\left(A^{-1}\right)=\frac{1}{\lambda\left(A\right)}.
\]
In general, $k=\left(k_1,k_2,\dots,k_m,k_{m+1}\right)$ represents a multiindex, and related to a multiindex
we define 
\[
\left|k\right|=k_1+k_2+\dots+k_m+k_{m+1}\quad \mbox{and}\quad k!=k_1!k_2!\dots k_m!k_{m+1}!.
\]
Also, we for
two multiindices $k$ and $k'$ we say that $k'<k$ if for every $j=1,2,\dots,m$, $k'_j\leq k_j$ an at least
for one index $l$, $k'_l<k_l$.

Let
\[
F:U\subset \mathbb{R}^{m+1}\longrightarrow \mathbb{R}^{m+1},
\quad
F\left(z\right)=\left(F_1\left(z\right),\dots,F_m\left(z\right),F_{m+1}\left(z\right)\right)^T,
\]
with $z=\left(x_1,\dots,x_m,t\right)^T$, be a smooth mapping.
Regarding differentiation
we write the column vector $F^{\left(k\right)}\left(x\right)$ as
\[
F^{\left(k\right)}\left(z\right)=
\left(\frac{\partial^{\left|k\right|}F_j}{\partial x_1^{k_1}\partial x_2^{k_2}\dots\partial x_m^{k_m}\partial t^{k_{m+1}}}\right)_{j=1,\dots,m},
\]
and its Jacobian matrix is defined as
\[
F'\left(z\right)=\left(\frac{\partial F_i}{\partial x_j}\left(z\right)\right)_{i,j=1,\dots,m,m+1}.
\]
As it is usual, we shall use the convention $z^{k}=x_1^{k_1}x_2^{k_2}\dots x_m^{k_m}t^{k_{m+1}}$
for a given multiindex $\left(k_1,\dots,k_m,k_{m+1}\right)$. 

\subsection{}

For the proof we recall Taylor's theorem in several variables with residue.
\begin{theorem}[Taylor's theorem]
Let $f$ be $k+1$ times continously differentiable on an open neighborhood
of $a$. Then we have that
\[
f\left(z\right)=f\left(a\right)+f'\left(a\right)\left(z-a\right)+\dots+
\sum_{\left|\beta\right|=k+1}R_{\beta}\left(z\right)\left(z-a\right)^{\beta},
\]
where
\[
R_{\beta}\left(z\right)=\frac{\left|\beta\right|}{\beta!}\int_0^1 \left(1-u\right)^{\left|\beta\right|-1}
D^{\beta}f\left(a+u\left(z-a\right)\right)\,du.
\]
\end{theorem}

We have a useful observation to make. First we have,
\[
\partial_{j} R_{\beta}\left(z\right)=\frac{\left|\beta\right|}{\beta!}\int_0^1\left(1-u\right)^{\left|\beta\right|-1}u
\left[\partial_j D^{\beta}f\left(a+u\left(z-a\right)\right)\right]\,dt,
\]
and hence
we have an estimate 
\begin{equation}
\label{Residuestimate}
\left|\partial_j R_{\beta}\left(z\right)\right|\leq 
\frac{1}{\beta! \left(\left|\beta\right|+1\right)}\max_{\left|\alpha\right|=k+2}
\left(\sup_{y\in B_{\left\|z\right\|}\left(a\right)}\left|F^{\left(\alpha\right)}\left(y\right)\right|\right).
\end{equation}
This observation will be useful below.

\section{Derivative estimates}

\begin{lemma}
\label{estimatesfromfundamentalsolution}
Let $u:\Omega\subset \mathbb{R}^{m+1} \longrightarrow \mathbb{R}$
be a solution to the heat equation. Assume that $\overline{B}_1\left(0,0\right)\subset \Omega$.
For each multiindex $k$, there exists a constant $C_k$, which only depends 
on $k$ and the dimension ($m+1$), such that the following holds for a solution to heat
equation
\[
\sup_{z\in B_\frac{1}{4}\left(0,0\right)}\left|\partial^k u\left(z\right)\right|
\leq C_k\sup_{z\in B_1\left(0,0\right)}\left|u\left(z\right)\right|.
\]
\end{lemma}
\begin{proof}
Let $H\left(x,t\right)$
be the fundamental solution for the heat equation.
Let $\left(x_0,t_0\right)\in B_{\frac{1}{4}}\left(0,0\right)$.
We first construct a cutoff function as follows. We fix $0\leq \varphi_1\leq 1$ which is 1
on $B_\frac{1}{8}\left(x_0\right)$ and whose support is contained
in $B_{\frac{1}{4}}\left(x_0\right)$. We fix $0\leq \varphi_2\leq 1$ which is 1 on
$\left[t_0-\frac{1}{64},t_0+\frac{1}{64}\right]$ and whose support
is contained in $\left(t_0-\frac{1}{16},t_0+\frac{1}{16}\right)$. Then define
\[
\varphi\left(x,t\right)=\varphi_1\left(x\right)\varphi_2\left(t\right).
\]
Let 
\[
L\left(\varphi u\right) = v.
\]
Then, we can write
\[
\varphi u\left(x,t\right)=\left<H,v\right>,
\]
where
\[
\left<H,v\right> = \int_{-\infty}^t\int_{-\infty}^{\infty}H\left(x-y,t-s\right)v\left(y,s\right)\,dy\,ds.
\]

We will concentrate on estimating derivatives with respect to the spatial variables,
since any time variable can be exchanged by a certain amount of them.

Let $\psi=\varphi\left(4x,16t\right)$, and
write
\[
H*v=\left(1-\psi\right)H*v+\psi H*v.
\]
First notice that $\psi H* v$ vanishes in $B_{\frac{1}{64}}\left(x_0,t_0\right)$; hence 
we only must obtain bounds 
on $\partial^{\alpha}\left[\left(1-\psi\right)H*v\right]$ for a multiindex $\alpha$.
Our task is then to estimate 
\[
\partial^{\alpha}\left(\left[\left(1-\psi\right)H\right]*v\right)
=
\left(\partial^{\alpha}\left[\left(1-\psi\right)H\right]\right)*v.
\]
Since by Leibniz formula
\[
\partial^{\alpha}\left[\left(1-\psi\right)H\right]=
\sum \binom{\alpha}{\beta}\partial^{\alpha-\beta}\left(1-\psi\right)\partial^{\beta}H,
\]
and for $\alpha\neq \beta$, $\partial^{\alpha-\beta}\left(1-\psi\right)$ vanishes in the 
ball $B_{\frac{1}{32^2}}\left(x_0,t_0\right)$, so we can write
\[
\partial^{\alpha}\left(\left[\left(1-\psi\right)H\right]*v\right)
=
\left(\left[\left(1-\psi\right)\partial^{\alpha}H\right]\right)*v+R,
\]
where $R$ vanishes in $B_{\frac{1}{32^2}}\left(x_0,t_0\right)$.

Thus we have for $z'\in B_{\frac{1}{32^2}}\left(z_0\right)$, $z_0=\left(x_0,t_0\right)$, 
\begin{eqnarray*}
\left|\partial^{\alpha}\left(\left[\left(1-\psi\right)H\right]*v\right)\left(z'\right)\right|&=&
\left|\left<v,\left[\left(1-\psi\right)\partial_x^{\alpha}H\right]\left(x-x_0,t-t_0\right)\right>\right|\\
&=&
\left|\left<\phi u,L^*_{x,t}\left[\left(1-\psi\right)\partial_x^{\alpha}H\right]\left(x-x_0,t-t_0\right)\right>\right|\\
&\leq&
M\sup_{z\in B_{1}\left(0,0\right)} \left|u\left(z\right)\right|
\sup_{\left|\sigma\right|\leq 2+\left|\alpha\right|,\frac{1}{32^2}\leq\left\|z-z_0\right\|\leq \frac{9}{8}}\left|\partial^{\sigma}H\left(z-z_0\right)\right|
\\
&=&
M\sup_{z\in B_{1}\left(0,0\right)} \left|u\left(z\right)\right|
\sup_{\left|\sigma\right|\leq 2+\left|\alpha\right|,\frac{1}{32^2}\leq\left\|z\right\|\leq \frac{9}{8}}\left|\partial^{\sigma}H\left(z\right)\right|,
\end{eqnarray*}
where $M$ is a constant independent of $H$ and $\alpha$ (To be more precise, $M$ only depends on the dimension $m$). 

\end{proof}

By rescaling we obtain the following general estimate for solutions to the heat equation.
\begin{lemma}
\label{cauchyestimates}
Let $u$ be a solution to the heat equation on an open set $\Omega\subset \mathbb{R}^{m+1}$, and
let $0<r<1$. Let $a$ be a point in $\Omega$ such that $\overline{B_{r}\left(a\right)}\subset \Omega$
There exists a constant $a_m$ such that
\[
\sup_{B_{\frac{r^2}{4}}\left(a\right)}\left|\partial^k u\left(x\right)\right|
\leq \frac{a_m}{r^{2\left|k\right|}}\sup_{B_{r}\left(a\right)}\left|u\left(x\right)\right|,
\]
for all multiindices $k$ such that $\left|k\right|=1,2$. Without loss of generality,
we may assume that $a_m\geq 1$.
\end{lemma}

\begin{proof}
By Lemma \ref{estimatesfromfundamentalsolution}, there is a $a_m$ such that
\[
\sup_{z\in B_{\frac{1}{4}}\left(a\right)}\left|\partial^k u\left(z\right)\right|
\leq a_m\sup_{z\in B_{1}\left(a\right)}\left|u\left(z\right)\right|,
\]
with $\left|k\right|=1,2$ (as we have fixed the sizes of the multiindex we can drop the dependence 
on $k$ for the constant given by Lemma \ref{cauchyestimates} and focus only on its dependence
on the dimension). 
On the other hand, if $u$ is a solution to the heat
equation $u\left(rx,r^2t\right)$ is also a solution to the heat equation. 
Using this rescaling and assuming $0<r<1$, we have the following. 
If $u$ is a solution of the heat equation in $B_{r^2}$, rescaling
\[
u_{r}\left(x,t\right)=u\left(r^2x,rt\right)
\]
is a solution to the heat equation in an ellipsoid of major semiaxis $1$ and minor semiaxis $r$.
This is contained in the ball of radius 1. So by previous estimates we have
\[
\sup_{B_\frac{1}{4}}\left|\partial^{\left|k\right|}u_r\right|\leq a_m\sup_{B_1}\left|u_r\right|=a_m\sup_{B_r}\left|u\right|,
\]
but, using the notation $k_x=\left(k_1,\dots,k_m\right)$ and $k_t=k_{m+1}$,
\[
\partial^{\left|k\right|}u_r=r^{2k_t+\left|k_x\right|}\partial^{\left|k\right|}u\left(rx,r^2t\right),
\]
and thus
\[
\sup_{B_\frac{r^2}{4}}\left|\partial^{\left|k\right|}u\right|\leq 
\frac{a_m}{r^{2k_t+\left|k_x\right|}}\sup_{B_r}\left|u\right|\leq \frac{a_m}{r^{2\left|k\right|}}\sup_{B_r}\left|u\right|.
\]
\end{proof}

\subsection{Heat Bochner-Takahashi $K$-mappings.}
As announced in the introduction,
we will consider maps 
\[F:B_1\left(0\right)\subset \mathbb{R}^{m+1}\longrightarrow\mathbb{R}^{m+1},\]
each of whose components satisfies the heat equation, and which satisfy the following estimate
\begin{equation}
\label{Takahashicondition}
\max_{\left\|z\right\|\leq r}\left\|F'\left(z\right)\right\|\leq K
\max_{\left\|z\right\|\leq r}\left|\mbox{det}\left(F'\left(z\right)\right)\right|^{\frac{1}{m+1}},
\quad \mbox{for all}\quad 0\leq r <1.
\end{equation}
We call $F$ a heat Bochner-Takahashi $K$-mapping.

As a normalisation
we impose that $\left|\mbox{det}\left(F'\left(0\right)\right)\right|=1$, and we will
assume without loss of generality that $F\left(0\right)=0$ and
that $F'$ is bounded on $\overline{B}_1\left(0\right)$. We will show that 
for this family of maps Bloch's theorem holds (below we explain why there is no 
need to worry about the case when $F'$ is not bounded).

We begin our proof just as before, by
picking a of positive numbers $r_j\in \left(0,1\right)$, $j=0,1,2,\dots$ as
follows. First, pick any $r_0=r_{\gamma}> 0$
and then choose $\gamma>1$ (so the choice of $r_0$ determines the choice of $\gamma$) such that
\begin{equation}
\label{definitionofbeta}
r_{\gamma}\prod_{j=1}^{\infty}\left(1+\gamma^{-j}\right)=1
\end{equation}
and construct maximal finite sequences $r_j$ and $\epsilon_j$, $j=0,\dots,l$ 
in the following way:

\medskip
 Denote
by $M\left(r_{j}\right)$ the maximum of $\left|\mbox{det}\left(F'\left(x\right)\right)\right|$
in the closed ball of radius $r_j$, that is
\[
M\left(r_{j}\right)=\max_{\left\|x\right\|\leq r_j}
\left|\mbox{det}\left(F'\left(x\right)\right)\right|.
\]
Once $r_0,r_1,\dots,r_{n-1}$ have been chosen, with each $r_j>r_{j-1}$, and
\[
M\left(r_j\right)^{\frac{1}{m+1}}=\gamma^4 M\left(r_{j-1}\right)^{\frac{1}{m+1}}, \quad j=1,2\dots,n-1,
\]
and $\epsilon_0,\epsilon_1,\dots,\epsilon_{n-2}$,
$n\geq 1$
such that
\[
r_j=\left(1+\epsilon_{j-1}\right)r_{j-1}, \quad j=1,2,\dots,n-1,
\]
we construct $r_{n}$ and $\epsilon_{n-1}$ as follows.
If there is {\bf no} $r_{n}>r_{n-1}$ with $r_n<1$ such that 
\begin{equation*}
\left(\frac{M\left(r_n\right)}{M\left(r_{n-1}\right)}\right)^{\frac{1}{m+1}}=\gamma^4,
\end{equation*}
we stop defining the $r$'s, i.e., we leave the sequence of $r$'s
as it is (its terms finally being $r_0,r_1,\dots,r_{n-1}$),
and define 
\[\epsilon_{n-1}=\dfrac{1}{r_{n-1}}-1,\]
and we have finished defining our sequences.
Otherwise, choose $r_n$ so that
\[
\left(\frac{M\left(r_{n}\right)}{M\left(r_{n-1}\right)}\right)^{\frac{1}{m+1}}=\gamma^4 \quad
\mbox{and then set}
\quad
\epsilon_{n-1}=\frac{r_n}{r_{n-1}}-1.
\]
The assumption on the boundedness of $F'$ on the closure of $B_1\left(0\right)$
implies that the construction eventually stops, leaving as a result two finite sequences. We shall denote the sequences thus constructed by
$r_0,r_1,\dots,r_l$ and $\epsilon_0,\epsilon_1,\dots,\epsilon_{l}$.

Notice then that the sequence $\epsilon_0,\epsilon_1,\dots,\epsilon_{l}$ satisfies
\begin{equation}
\label{definitionepsilon}
r_0\prod_{j=0}^{l}\left(1+\epsilon_{j}\right)= 1.
\end{equation}
Observe also the following important two facts:
we have that $\left(\dfrac{M\left(1\right)}{M\left(r_{l}\right)}\right)^{\frac{1}{m+1}}\leq \gamma^4$,  
\[
\left(\dfrac{M\left(r_{n+1}\right)}{M\left(r_{n}\right)}\right)^{\frac{1}{m+1}}=\gamma^4 \quad \mbox{for}\quad 0\leq n\leq l-1,
\]
and also that at least for one $k$, it must hold that $\epsilon_k\geq \gamma^{-\left(k+1\right)}$.

It might serve as a clarification to the reader 
to show a couple of situations on how the construction described above might turn out. First, it could happen that the two sequences
defined above contain only one element. Indeed, once $r_0$ is chosen, if its 
corresponding $\gamma$ is such that 
$M\left(1\right)^{\frac{1}{m+1}}\leq \gamma^4$, then the sequence of $r$'s would only consist
of $r_{0}$ (and in this case $l=0$), and the sequence of $\epsilon$'s only of $\epsilon_0$, and we would actually 
have that
\[
\epsilon_0=\frac{1}{r_0}-1,
\]
so the whole construction might stop at the first step (in other words, it might not be possible to find $\beta_1$).

Another situation that may arise is, for instance, that once $r_0$ has been chosen
it occurs that its corresponding $\gamma$
satisfies $\gamma^4 M\left(r_0\right)^{\frac{1}{m+1}}<M\left(1\right)^{\frac{1}{m+1}}\leq M\left(r_0\right)^{\frac{1}{m+1}}\gamma^8$. Then the sequences of $r$'s 
and $\epsilon$'s would consist each of only two terms (in this case $m=1$), $r_0$ and $r_1$ with
$\left(M\left(r_1\right)/M\left(r_0\right)\right)^{\frac{1}{m+1}}=\gamma^4$, and
\[
\epsilon_0=\frac{r_1}{r_{0}}-1,\quad
\epsilon_1=\frac{1}{r_{1}}-1.
\]
For convenience, we define $r_{l+1}=1$.

Before we continue, as the assumption on the boundedness of $F'$ might be of concern to the reader,
we want to point out the following. First, if we are only interested in proving the
existence 
of a bound from below for the radius of a univalent disk covered by a member of the family under consideration,
we just have to proceed with our arguments in a ball centered at $\left(0,0\right)$ of radius strictly
smaller than 1. On the other hand, 
if $F'$ is not bounded the construction described above would not stop,
but the arguments below work just the same; all
we actually need, if $F'$ happens to be unbounded, is that the closures of the balls 
involved in the arguments below are contained in $B_1\left(0,0\right)$, 
and it is easy to see that it happens if $F'$ is unbounded.

We want to solve the equation
\[
w=F\left(z\right).
\]
To proceed, we let $\beta_n\in \overline{B}_{r_n}\left(0,0\right)$ a
point where $M\left(r_n\right)$ is reached.
By Taylor's theorem, solving the previous equation is equivalent to solving
\[
w=F\left(\beta_n\right)+F'\left(\beta_n\right)\left(z-\beta_n\right)
+\sum_{\left|\alpha\right|=2}R_{\alpha}\left(z\right)\left(z-\beta_n\right)^{\alpha}.
\]
This is equivalent to the following 
fixed point problem:
\begin{equation}
\label{fixedpointproblem1}
z=\left[F'\left(\beta_n\right)\right]^{-1}\left(w-F\left(\beta_n\right)\right)
+\beta_n-
\sum_{\left|\alpha\right|=2}R_{\alpha}\left(z\right)F'\left(\beta_n\right)^{-1}\left(z-\beta_n\right)^{\alpha}.
\end{equation}
Let us define
\[
g_w\left(z\right):=F'\left(\beta_n\right)^{-1}\left(w-F\left(\beta_n\right)\right)+\beta_n
-\sum_{\left|\alpha\right|=2}R_{\alpha}\left(z\right)F'\left(\beta_n\right)^{-1}\left(z-\beta_n\right)^{\alpha}.
\]

The main idea now is to show that there is a disk, centered at $F\left(\beta_n\right)$, call it $D$,
such that if $w\in D$, we can restrict $g_w$ to an $m+1$-dimensional closed ball
centered at $\beta_n$, $n\leq l$, say $D'$, whose radius is independent of $w$, so that 
$g_w:D'\longrightarrow D'$, and so that it is a contraction. Then Banach's
Contraction Mapping Principle can be applied now to show that (\ref{fixedpointproblem1}) has a unique solution, and 
this shows that for every $w\in D$ there is a unique $z\in D'$ such that
$F\left(z\right)=w$.
From this we can conclude that when restricted to $D'\cap F^{-1}\left(D\right)$, $F$ is one to one and onto $D$,
and that the radius of $D$ is a bound from below for Bloch's constant.

Before we continue we must observe the following. 
A map $G:\Omega\subset \mathbb{R}^{m+1}\longrightarrow \mathbb{R}^{m+1}$, $\Omega$ convex, 
such that each row $G'_k$ of its Jacobian matrix satisfies
\[\left\|G'_k\left(x\right)\right\|\leq \dfrac{1-\sigma}{\sqrt{m+1}}, \quad\sigma>0, \]
for all $x\in \Omega$, is a contraction.
So we impose a condition on $F$ to make sure that $g_w$ with domain 
$D'=\overline{B}_{\eta}\left(\beta_n\right)\subset B_{1}\left(0,0\right)\subset \mathbb{R}^{m+1}$ is a contraction, namely
\begin{eqnarray*}
&\sum_{\left|k\right|= 2}\dfrac{\left\|F'\left(\beta_n\right)^{-1}\right\|
\max_{x\in \overline{B}_{\frac{\left(\epsilon_n r_n\right)^2}{4}}\left(\beta_n\right)}\left|F^{\left(k\right)}\left(z\right)\right|}{\left(k-1\right)!}
\eta^{\left|k\right|-1}&\\
&+&\\
&\sum_{\left|k\right|= 2}\dfrac{\left\|F'\left(\beta_n\right)^{-1}\right\|
\max_{x\in \overline{B}_{\frac{\left(\epsilon_n r_n\right)^2}{4}}\left(\beta_n\right)}\left|F^{\left(k+1\right)}\left(z\right)\right|}
{k!\left(\left|k\right|+1\right)}
\eta^{\left|k\right|}&\\
&\leq&\\
&\dfrac{1-\sigma}{\sqrt{m+1}}.&
\end{eqnarray*}
In the inequality above keep in mind that $k$ is a multiindex, and that $k-1$ is a shorthand for a multiindex $k'$ such that
$k'<k$ and $\left|k\right|-\left|k'\right|=1$.
Also, notice that in the previous estimate we are assuming that $\eta\leq \left(\epsilon_n r_n\right)^2/4$,
(and observe that we have also used estimate (\ref{Residuestimate})).

Using Lemma \ref{cauchyestimates}, and the fact that for any matrix $A=\left(a_{ij}\right)$
the inequality $\left|a_{ij}\right|\leq \left\|A\right\|$ holds, 
the previous inequality
can be replaced by the (stronger) condition (we need that $\epsilon_nr_n/a_m\leq 1$ below,
but this will be so as long as $a_m\geq 1$)
\begin{equation}
\label{contractionconditionscvm}
2^{2}\times 2^{m+1}
\frac{\left(a_m\right)^{\left|k\right|-1}\left\|F'\left(\beta_n\right)^{-1}\right\|\max_{x\in
\overline{B}_{\epsilon_nr_n}\left(\beta_n\right)}\left\|F'\left(z\right)\right\|}
{\left(\epsilon_{n}r_n/a_m\right)^{4}}
\eta^{\left|k\right|-1}\leq \frac{1-\sigma}{\sqrt{m+1}},
\end{equation}
where $a_m\geq 1$ is a dimensional constant that only depends on the $m+1$-dimensional heat kernel.
The factor $2^2$ has been included so that at the end of our estimates we guarantee that
$\eta\leq \left(\epsilon_n r_n\right)^2/4$.

In what follows, we shall use the notation $\lambda_F\left(z\right)$ to indicate $\lambda\left(F'\left(z\right)\right)$.
Now, 
we replace (\ref{contractionconditionscvm}) by the stronger inequality
\begin{equation}
\label{contractionconditionscvm2}
2^{m+3}
\frac{\left\|F'\left(\beta_n\right)^{-1}\right\|\max_{\left\|z\right\|\leq r_{n+1}}\left\|F'\left(z\right)\right\|}
{\left(\epsilon_{n}r_n/a_m\right)^{4}}
\eta\leq \frac{1-\sigma}{\sqrt{m+1}}.
\end{equation}
On the other hand, to make sure that
$g_w:\overline{B}_{\eta}\left(\beta_n\right)\longrightarrow \overline{B}_{\eta}\left(\beta_n\right)$, we can
estimate from (\ref{fixedpointproblem1})
\begin{eqnarray*}
\left\|F'\left(\beta_n\right)^{-1}\right\|\left\|w-F\left(\beta_n\right)\right\|
&\leq& \left\|z-\beta_n\right\|-
\left\|\sum_{\left|k\right|=2}R_{k}\left(z\right)F'\left(\beta_n\right)^{-1}\left(z-\beta_n\right)^{k}\right\|\\
&\leq& \eta - \sum_{k=2}\left\|F'\left(\beta_n\right)^{-1}\right\|\left|R_k\left(z\right)\right|
\eta^{\left|k\right|}\\
&\leq&\eta-\eta\sum_{\left|k\right|= 2}\left\|F'\left(\beta_n\right)^{-1}\right\|
\frac{\max_{x\in \overline{B}_{\frac{\left(\epsilon_n r_n\right)^2}{4}}\left(\beta_n\right)}\left|F^{\left(k\right)}\left(z\right)\right|}{k!}
\eta^{\left|k\right|-1}\\
&\leq& \eta-\left(1-\sigma\right)\eta=\sigma\eta.
\end{eqnarray*}
Therefore if
\[
\left\|w-F\left(\beta_n\right)\right\|\leq \sigma \eta \lambda_F\left(\beta_n\right),
\]
then $g_w$ sends the ball $\overline{B}_{\eta}\left(\beta_n\right)$ to itself. Notice that this give an estimate 
for the radius of the disk $D$ we mentioned above, and thus $\sigma \eta \lambda_F\left(\beta_n\right)$ would
give an estimate for the radius of a schlicht disk covered by $F$.

Next, we estimate $\eta$ in terms of $\epsilon_{n}$ and $r_n$. The estimate we need for 
$\eta$
is a consequence of (\ref{contractionconditionscvm2}), so let us work on this inequality.
First, we estimate
\begin{eqnarray*}
\left\|F'\left(\beta_{n}\right)^{-1}\right\|\max_{\left\|z\right\|\leq r_{n+1}}\left\|F'\left(z\right)\right\|
&=& \frac{\sqrt{m+1}\max_{\left\|z\right\|\leq r_{n+1}}\left\|F'\left(z\right)\right\|}{\lambda_F\left(\beta_n\right)}\\
&\leq&
\frac{\sqrt{m+1}K\left|\mbox{det}\left(F'\left(\beta_{n+1}\right)\right)\right|^{\frac{1}{m+1}}}{\lambda_F\left(\beta_n\right)}\\
&\leq&\frac{\sqrt{m+1}K^{m+2}\left|\mbox{det}\left(F'\left(\beta_{n+1}\right)\right)\right|^{\frac{1}{m+1}}}
{\left|\mbox{det}\left(F'\left(\beta_{n}\right)\right)\right|^{\frac{1}{m+1}}}\\
&=&\sqrt{m+1}K^{m+2}\gamma^4,
\end{eqnarray*}
where we have used the following fact, which holds
for Bochner-Takahashi $K$-mappings because of the choice of $\beta_n$ (and whose proof we postpone
to the end of the section):
\begin{equation}
\label{Wusmalleigen}
\lambda_{F}\left(\beta_{n}\right)\geq \frac{1}{K^{m+1}}\left|\mbox{det}\left(F'\left(\beta_n\right)\right)\right|^{\frac{1}{m+1}}.
\end{equation}
So, from (\ref{contractionconditionscvm2}), we obtain:
\[
2^{m+3}
\sqrt{m+1}K^{m+2} \gamma^4
\dfrac{\eta}{\left(\epsilon_{n}r_{n}/a_m\right)^4}\leq 
\frac{1-\sigma}{\sqrt{m+1}},
\]
and hence
\[
\eta \leq \frac{1-\sigma}{2^{m+3}\left(m+1\right) K^{m+2} \gamma^4}\left(\frac{\epsilon_nr_n}{a_m}\right)^4.
\]
Observe that $\eta\leq \left(\epsilon_n r_n\right)^2/4$ (of course 
notice that $\epsilon_nr_n<1$ and $a_m\geq 1$), as we need it to be.
This estimate on $\eta$ in turn implies the following estimate from below for 
the radius of a univalent disk covered by $F$:
\[
\frac{\left(1-\sigma\right)\sigma}{2^{m+3}\left(m+1\right) K^{m+2} \gamma^4}\left(\frac{\epsilon_nr_n}{a_m}\right)^4 \lambda_F\left(\beta_n\right).
\]
Using (\ref{Wusmalleigen}), and that by construction there is an $n$ for which the inequality
\[M\left(r_n\right)^{\frac{1}{m+1}}\epsilon_{n}^4\geq 
\frac{M\left(r_{\gamma}\right)^{\frac{1}{m+1}}}{\gamma^4}\]
holds, this estimate becomes (using (\ref{Wusmalleigen}) again)
\begin{equation}
\label{blochwu1}
\frac{\left(1-\sigma\right)\sigma}{2^{m+3}\left(m+1\right) K^{2m+3} \gamma^4}\left(\frac{r_{\gamma}}{a_m}\right)^4 
\frac{M\left(r_{\gamma}\right)^{\frac{1}{m+1}}}{\gamma^4}.
\end{equation}

Now we go for {\bf the second part of the argument}. Consider the fixed point problem
\begin{equation}
\label{FPPOSCV}
z=\left(F'\left(0\right)^{-1}\right)w-\sum_{\left|k\right|= 2}F'\left(0\right)^{-1}R_{k}\left(z\right)z^{k}.
\end{equation}
In this case, imposing to the left hand side of (\ref{FPPOSCV}) to be a contraction, using Lemma \ref{cauchyestimates}, 
and Bloch-Takahashi's condition (\ref{Takahashicondition}), we obtain that
\begin{equation}
\label{Takahashicondition1}
2^{m+3}
\sqrt{m+1}K M\left(r_{\gamma}\right)^{\frac{1}{m+1}}
\dfrac{\eta}{\left(r_{\gamma}/a_m\right)^4}\leq \frac{\left(1-\sigma\right)}{\sqrt{m+1}}\lambda_F\left(0\right).
\end{equation}
Proceeding as before, we arrive at the following estimate from below for the radius of a schlicht disk covered by $F$
\[
\frac{\sigma\left(1-\sigma\right)}{2^{m+3} \left(m+1\right)}\frac{1}{K M\left(r_{\gamma}\right)^{\frac{1}{m+1}}}
\left(\frac{r_{\gamma}}{a_m}\right)^4\lambda_{F}\left(0\right)^2
\]
which, using (\ref{Wusmalleigen}) with $\beta_n=0$ and the 
normalisation, gives an estimate from below for the radius of a schlicht disk covered by $F$, namely,
\begin{equation}
\label{blochwu2}
\frac{\sigma\left(1-\sigma\right)}{\left(m+1\right) 2^{m+3}}\frac{1}{K^{2m+3} M\left(r_{\gamma}\right)^{\frac{1}{m+1}}}
\left(\frac{r_{\gamma}}{a_m}\right)^4.
\end{equation}
Before we continue, observe that if the maximum of $\left|\mbox{det}\left(F'\left(x\right)\right)\right|$
on the unit ball
is 1, then
the first part of the proof becomes unnecessary. Indeed, $r_{\gamma}$ can 
be chosen as close to the unit circle as wanted, 
and hence we would get that $F$ covers a schlicht 
disk of radius
\[
\frac{\sigma\left(1-\sigma\right)}{\left(m+1\right)2^{m+3}}\frac{1}{K^{2m+3}}
\frac{1}{a_m^4}.
\]
In any case, notice that (\ref{blochwu1}) is better than (\ref{blochwu2}) when
\[
M\left(r_{\gamma}\right)^{\frac{1}{m+1}}\geq \gamma^4,
\]
whereas (\ref{blochwu2}) is better than (\ref{blochwu1}) when the opposite inequality holds; but before we use this fact 
to give an estimate from below for the radius of a schlicht disk covered by a heat 
Bochner-Takahashi $K$-mapping,
we must estimate $r_{\gamma}$:
starting from (\ref{definitionofbeta}), we obtain
\[
\log r_{\gamma}=-\sum_{j\geq 1}\log\left(1+\gamma^{-j}\right),
\]
and then by Taylor's theorem
\[
\log\left(1+\gamma^{-j}\right)\leq \gamma^{-j}-\frac{1}{2}\gamma^{-2j}+\frac{1}{3}\gamma^{-3j},
\]
so we have
\[
\log r_{\gamma}\geq -\frac{1}{\gamma-1}+\frac{1}{2}\frac{1}{\gamma^2-1}-\frac{1}{3}\frac{1}{\gamma^3-1},
\]
that is,
\[
r_{\gamma}\geq e^{-\frac{1}{\gamma-1}+\frac{1}{2}\frac{1}{\gamma^2-1}-\frac{1}{3}\frac{1}{\gamma^3-1}}.
\]

From (\ref{blochwu1}) and (\ref{blochwu2}),
the observation on when between these
two estimates is better than the other, and the estimate from below for $r_{\gamma}$,
we conclude that Bloch's constant for harmonic Bochner-Takahashi $K$-mappings is bounded from below
by
\[
\frac{\sigma\left(1-\sigma\right)}{\left(m+1\right)2^{m+3}}\frac{1}{K^{2m+3}}
\left(\frac{r_{\gamma}}{a_m\gamma}\right)^4
\geq \frac{\sigma\left(1-\sigma\right)}{\left(m+1\right)2^{m+3}}\frac{1}{K^{2m+3}}
\left(\frac{e^{-\frac{1}{\gamma-1}+\frac{1}{2}\frac{1}{\gamma^2-1}-\frac{1}{3}\frac{1}{\gamma^3-1}}}{a_m\gamma}\right)^4.
\]
This proves Theorem \ref{BlochTakahashi} giving, after playing around a bit with $\sigma$ and $\gamma$, 
the bound from below
\[
\frac{0.22^4}{2^{m+5}\left(m+1\right)}\frac{1}{a_m^4 K^{2m+3}}
\]
for Bloch's constant of heat Bochner-Takahashi $K$-mappings.

\subsection{Last Remarks}
Here we prove inequality (\ref{Wusmalleigen}). Since $\beta_n$ is a point where
$M\left(r_n\right)$ is reached, we have by the Takahashi-Bochner condition
\begin{eqnarray*}
\left\|F'\left(\beta_n\right)\right\|&\leq& K\left|\mbox{det}\left(F'\left(\beta_n\right)\right)\right|^{\frac{1}{m+1}}\\
&\leq&K \lambda_F\left(\beta_n\right)^{\frac{1}{m+1}}\Lambda_F\left(\beta_n\right)^{\frac{m}{m+1}},
\end{eqnarray*}
from which we obtain $\Lambda_F\left(\beta_n\right)\leq K^m \lambda_F\left(\beta_n\right)$.
This inequality can be used to obtain
\begin{eqnarray*}
\Lambda_F\left(\beta_n\right)&\leq&K\left|\mbox{det}\left(F'\left(\beta_n\right)\right)\right|^{\frac{1}{m+1}}\\
&\leq&
K\left[\lambda_F\left(\beta_n\right)\right]^{\frac{1}{m+1}}\left[K^m\lambda_F\left(\beta_n\right)\right]^{\frac{m}{m+1}}
=K^{m+1}\lambda_F\left(\beta_n\right),
\end{eqnarray*}
from which we readily get (\ref{Wusmalleigen}).

\bibliographystyle{amsplain}

\end{document}